\documentclass[12pt]{amsart}

\usepackage{amsfonts, amstext, amsmath, amsthm, amscd, amssymb, graphicx, epstopdf}
\usepackage{epsfig, graphics, psfrag, enumerate}
\usepackage{color}
\usepackage{verbatim}
\let\oldmarginpar\marginpar
\renewcommand\marginpar[1]{\oldmarginpar[\raggedleft\footnotesize #1]%
{\raggedright\footnotesize #1}}

\newtheorem{theorem}{Theorem}[section]

\newtheorem{corollary}[theorem]{Corollary}

\newtheorem{problem}[theorem]{Problem}

\newtheorem{define}[theorem]{Definition}

\theoremstyle{definition}

\newcommand{\NN}{{\mathbb{N}}}

\newcommand{\QQ}{{\mathbb{Q}}}

\newcommand{\bdy}{{\partial}}

\newcommand{\abs}[1]{{\left\vert #1 \right\vert}}

\newcommand{\GA}{{\mathbb{G}_A}}
\newcommand{\GB}{{\mathbb{G}_B}}

\newcommand\no[1]{}

\newtheorem*{namedtheorem}{\theoremname}
\newcommand{\theoremname}{testing}

\def\la{\langle}
\def\ra{\rangle}

\def\be { \begin{equation} }
\def\ee { \end{equation} }

\begin{document}

\title[]{A Jones slopes characterization of adequate knots}

\author[]{Efstratia Kalfagianni}
\address{Department of Mathematics, Michigan State University, East Lansing, MI, 48824}
\email{kalfagia@math.msu.edu}
\bigskip

\bigskip

\begin{abstract} We establish a characterization of adequate knots in terms of the degree of their colored Jones polynomial.
We show that, assuming the Strong Slope conjecture, our characterization can be reformulated  in terms of  ``Jones slopes" of  knots and the 
essential surfaces that realize the slopes.
For alternating knots the reformulated characterization follows by  recent work of J. Greene and J. Howie.
\bigskip

\bigskip

\bigskip

\noindent {2010 {\em Mathematics Classification:} {\rm  57N10, 57M25.}\\

}

\end{abstract}

\bigskip

\bigskip
\thanks {\today}
\thanks{Supported in part by NSF grants DMS--1105843 and DMS--1404754}

\maketitle

\section{Introduction} Adequate knots form a large class of knots that  behaves well with respect to Jones-type knot invariants and has nice topological and geometric properties
\cite{ Abe, armond, dasbach-lin:head-tail,  FKP,  fkp:filling,  FKP-guts, fkp:survey, fkp:qsf,  FKP-semi, lick-thistle, thi:adequate}.  Several well known classes of knots  are adequate; these include all alternating knots and Conway sums of strongly alternating tangles.
The definition of adequate knots, much like that of  alternating knots, requires the existence of knot diagram of particular type
 (see Definition \ref{defi:adequate}).  The work of Kauffman \cite{Ka}, Murasugi  \cite{murasugitait} and   Thistlethwaite \cite{Thistle}
 that settled the Tait
conjectures, provided a characterization of alternating knots in terms of the degree of the Jones polynomial: It showed that a knot is alternating precisely when the degree span
of its Jones polynomial determines the crossing number of the knot.
In this note we obtain a similar  characterization for adequate knots in terms of the degree span of colored Jones polynomial.
Roughly speaking, we show that adequate knots are characterized by the property that the degree of their colored Jones polynomial determines two basic topological invariants: the crossing number and the Turaev genus.

To state our results, recall that the colored Jones polynomial of a knot $K$ is a collection 
of Laurent polynomials 
$$ \{ J_K(n):=J_{K}(n, t)\ | \ n=1, 2,...\},$$ 
in a variable $t$ such that we have $J_{K}(1, t)=1$ and $J_{K}(2, t)$ is the ordinary Jones polynomial of $K$.
Throughout the paper we will use the normalization adapted in \cite{Effie-Anh-slope}; see Section 2 for more details.
 Let  $d_+[J_{K}(n)]$ and  $d_-[J_{K}(n)]$ denote  the maximal and minimal degree of $J_{K}(n, t)$ in $t$.
Garoufalidis \cite{ga-quasi} showed that, given a knot $K$ there is a number $n_K>0$ such that, for  $n>n_K$, 
  we have
 $$d_+[J_{K}(n)]-d_-[J_{K}(n)]=  s_1(n) n^2 + s_2(n) n  + s_3(n),$$
 \noindent where, for $i=1,2,3$, $s_i:\NN \to \QQ$ is a periodic function  with integral period.
 
 Given a knot diagram $D=D(K)$ one can define its  Turaev genus $g_T(D)$; see Section 4 for details.
 The Turaev genus of a knot $K$, denoted by $g_T(K)$,  is defined to be the minimum  $ g_T(D)$ over all knot diagrams representing $K$.
 Although the original definition of the Turaev genus is based on Kauffman states of knot diagrams \cite{turaevs,DFK}, the work of
 Armond, Druivenga and  Kindred \cite{ADK}  implies that it can be defined purely in terms of certain projections of knots on  certain Heegaard surfaces
 of ${\bf S}^3$.

Our main result is  the following.
 \begin{theorem} \label{mainintro}  For a knot $K$  let  $c(K)$  and  $g_T(K)$ denote the crossing number and the Turaev genus of $K$, respectively.
The knot $K$ is adequate if and only if, for  some $n>n_K$, we have
\begin{equation}
\label{first}
s_1(n)=c(K) /2 \ \ \  {\rm and } \ \ \  s_2(n) =1-g_T(K)-c(K)/2.
\end{equation}
Furthermore, every diagram of $K$ that realizes $c(K)$ is adequate and it also realizes $g_T(K)$.
 \end{theorem}
 
 Some ingredients for the proof of Theorem \ref{mainintro}   are a result of Lee \cite{lee, leethesis} on upper bounds on the degree of the  colored Jones polynomial and a result of Abe \cite{Abe} on the Turaev  genus of adequate
 knots.  
 
 It is known that a knot is alternating precisely when  $g_T(K)=0$.
As a corollary of  Theorem \ref{mainintro}  and its proof we have the following.
 \begin{corollary} \label{algorithm} Let the notation and setting be as above.  A knot $K$ is alternating if and only if, for  some $n>n_K$ we have
 \begin{equation}
 \label{alt}
2s_1(n)+2s_2(n)=2 \ \ \ {\rm and} \ \ \  2s_1(n)=c(K).
\end{equation}
Furthermore, every diagram of $K$ that realizes $c(K)$ is alternating.
\end{corollary}

The degree of the colored Jones polynomial is conjectured to contain information about essential surfaces in knot complements.
The {\emph Strong Slope Conjecture} that was  stated by the author and Tran in  \cite{Effie-Anh-slope} and refines the {\emph Slope Conjecture} of Garoufalidis \cite{ga-slope}, asserts 
 that the cluster points of the function $s_1$ are boundary slopes of the knot $K$ and that  the cluster points of $s_2$ 
 predict the topology of essential surfaces
 in the knot complement realizing these boundary slopes. The cluster points of $s_1$ are called {\emph Jones slopes} of $K$. See next section for more details.
 Assuming the Strong Slope Conjecture,  Theorem \ref{mainintro} leads to a characterization of adequate knots in terms of Jones slopes and essential spanning surfaces (see Theorem \ref{mainjones }). In particular, assuming the Strong Slope Conjecture, Corollary \ref{algorithm} can be reformulated as follows:
A knot $K$ is alternating if and only if it
 admits Jones slopes $s, s^{*}$, that are  realized by
 essential spanning surfaces $S$, $S^{*}$, such that 
 
 \begin{equation}
 \label{carjsp}
 (s-s^{*} )/2+\chi(S)+\chi(S^{*})=2 \ \ \ {\rm and} \ \ \  s-s^{*} =2c(K).
\end{equation}

The Strong Slope Conjecture is known for adequate knots \cite{Effie-Anh-slope}; the proof shows that alternating knots satisfy equations \ref{carjsp}.
Conversely, recent work of Howie \cite{howie}  implies that knots that satisfy equation \ref{carjsp}  are alternating, providing additional evidence supporting the conjecture. 
More specifically, Howie  \cite{howie}  and independently  Greene  \cite{greene} obtained
intrinsic topological characterizations of alternating knots in terms of essential spanning surfaces and gave normal surface theory  algorithms to recognize the  alternating property.
In particular, 
 \cite{howie} shows  that a non-trivial knot $K$ is alternating if and only if it admits essential spanning surfaces $S$, $S^{*},$ with boundary slopes  $s, s^{*}$, such that
 
 \begin{equation}
 \label{Altsp}
(s-s^{*})/2+\chi(S)+\chi(S^{*})=2.  
\end{equation}
\vskip 0.06in

Thus if $K$ is a knot that satisfies equations \ref{carjsp}, then it satisfies \ref{Altsp} and thus $K$ is alternating. 
The results of this paper, and in particular  Theorem \ref{mainjones },   and our discussion above, motivate the following 
problem.
\begin{problem} Show that a knot $K$ is adequate if and only if it admits Jones slopes $s, s^{*}$, that are  realized by
 essential spanning surfaces $S$, $S^{*}$, such that 
\begin{equation}
s-s^*= 2c(K) \ \  {\rm and} \ \      \chi(S)+ \chi(S^{*})+c(K) =2-2g_T(K).
\end{equation}
\end{problem}

We should point out that Theorem  \ref{mainintro}  and Corollary \ref{algorithm} also hold for links. 
On the other hand,  at this writing,  the picture of the relations between the degree of colored Jones polynomials and boundary slopes is better developed for knots.
 For this reason, and for simplicity of exposition,  we have chosen to only discuss knots throughout this note.

This paper is organized as follows: In Section 2 we recall the definition of the colored Jones polynomial and above mentioned conjectures from \cite{ga-slope, Effie-Anh-slope}.
In Section 3 we recall definitions and background about adequate knots. In Section 4 first we prove a stronger version of Theorem \ref{main}.
Then, using the fact that the Strong Slopes Conjecture is known for adequate knots (see Theorem \ref{thm:essential}) we reformulate Theorem \ref{main}  in terms of spanning knot surfaces
(Theorem \ref{mainjones }).
In Section 5 we discuss the special case of alternating knots and compare equations \ref {alt}, \ref{carjsp} and \ref{Altsp} above.

\section{The colored Jones polynomial} We briefly recall the definition of the colored Jones polynomial in terms  Chebyshev polynomials. For more details the reader
is referred to \cite{Li}.

   For $n \ge 0$,  the Chebyshev polynomials of the second kind $S_n(x)$,  are defined recursively as follows:
\begin{equation}
\label{chev}
S_{n+2}(x)=xS_{n+1}(x)-S_{n}(x), \quad S_1(x)=x, \quad S_0(x)=1.
\end{equation}

Let $D$ be a diagram of a knot $K$. For an integer $m>0$, let $D^m$ denote the diagram obtained from $D$ 
by taking $m$ parallels copies of $K$. This is the $m$-cable of $D$ using the blackboard framing; if $m=1$ then $D^1=D$. 
Let $\la D^m \ra$ denote the Kauffman bracket of $D^m$: this is a Laurent polynomial over the integers in a variable $t^{-1/4}$ 
normalized so that $\la \text{unknot} \ra = -(t^{1/2}+t^{-1/2})$.  Let $c_+(D)$ and $c_-(D)$ denote the number of positive and negative crossings in $D$, respectively.
Also let
$c=c(D)=c_+(D) + c_-(D)$ denote the crossing number and $w=w(D)=c_+(D) - c_-(D)$ denote the writhe of $D$.

For $n>0$, we define 
$$J_K(n):=  ( (-1)^{n-1} t^{(n^2-1)/4} )^w (-1)^{n-1}  \la S_{n-1}(D)\ra$$
where $S_{n-1}(D)$ is a linear combination of blackboard cablings of $D$, obtained via equation \eqref{chev}, and the notation $\la S_{n-1}(D) \ra$ means extend the Kauffman bracket linearly. That is, for diagrams $D_1$ and $D_2$ and scalars $a_1$ and $a_2$, $\la a_1 D_1 + a_2 D_2 \ra = a_1 \la D_1 \ra + a_2 \la D_2 \ra$.
We have
$$J_{\text{unknot}}(n) = \frac{t^{n/2}- t^{-n/2}}{t^{1/2} -t^{-1/2}}.$$ 

For a knot  $K \subset S^3$, let $d_+[J_{K}(n)]$ and  $d_-[J_{K}(n)]$ denote  the maximal and minimal degree of $J_{K}(n)$ in $t$.

Garoufalidis \cite{ga-quasi}  showed that
the degrees $ d_+[J_{K}(n)] $ and $ d_-[J_{K}(n)] $ are quadratic {\em quasi-polynomials}.
This means that, given a knot $K$, there is $n_K\in \NN$ such that  for all $n>n_K$ we have
 $$4 \, d_+[J_{K}(n)] =  a(n) n^2 + b(n) n  + c(n)\ \ \  {\rm and}  \ \ \   4 \, d_-[J_{K}(n)] =  a^{*}(n) n^2 + b^{*}(n) n  + c^{*}(n),$$
 where the coefficients are periodic functions from $\NN $ to $\QQ$ with integral period. By taking  the least common multiple of the periods
 of these coefficient  functions we get
 a common period. This common period of the coefficient functions is called the {\em Jones period} of $K$.

For a sequence $\{x_n\}$, let $\{x_n\}'$ denote the set of its cluster points. 

\begin{define} \label{jslopes} The elements of the sets 
$$js_K:= \left\{ 4n^{-2}d_+[J_K(n)]  \right\}' \quad
 \mbox{and} \quad js^*_K:= \left\{ 4n^{-2}d_-[J_K(n)] \right\}' $$
 are called {\em Jones slopes} of $K$.  
 \end{define}
 
 Given a knot $K\subset S^3$, let
  $n(K)$ denote a tubular neighborhood of
$K$ and let $M_K:=\overline{ S^3\setminus n(K)}$ denote the exterior of
$K$. Let $\langle \mu, \lambda \rangle$ be the canonical
meridian--longitude basis of $H_1 (\bdy n(K))$.   A properly embedded surface $(S, \bdy S) \subset (M_K,
\bdy n(K))$, is called essential if it's $\pi_1$-injective and it is is not a boundary parallel annulus.
An element $a/b \in
{\QQ}\cup \{ 1/0\}$ is called a \emph{boundary slope} of $K$ if there
is an essential surface $(S, \bdy S) \subset (M_K,
\bdy n(K))$, such that  $\bdy S$ represents $a \mu + b \lambda \in
H_1 (\bdy n(K))$.  Hatcher showed that every knot $K \subset S^3$
has finitely many boundary slopes \cite{hatcher}. The {\em Slope Conjecture} \cite[Conjecture 1.2]{ga-slope}, asserts that the Jones slopes of any knot $K$ are 
 boundary slopes.
\begin{define} \label{jchar} 
Let  $\ell d_+[J_K(n)]$ denote the linear term of $d_+[J_K(n)]$ and let
$$jx_K:= \left\{ 2n^{-1}\ell d_+[J_K(n)]  \right\}'=\left\{b_K(n)\right\}'\  {\rm and} \    jx^{*}_K:= \left\{ 2n^{-1}\ell d_-[J_K(n)]  \right\}'=\left\{b^{*}_K(n)\right\}'.$$
\end{define}
The {\em  Strong Slope Conjecture} \cite[Conjecture 1.6]{Effie-Anh-slope}, asserts that given
a Jones slope of $K$, say
 $a/b\in js_K$, with $b>0$ and $\gcd(a, b)=1$, there is an essential surface $S\subset M_K$, with $\abs{\partial S}$  boundary components, 
and such that each component of $\partial S$ has slope $a/b$ and
$$\frac{2\chi(S)}{{\abs{\partial S} b}} \in jx_K.$$

Similarly,  given  $a^{*}/ b^{*}\in js^{*}_K$, with $b^{*}>0$ and $\gcd(a^{*}, b^{*})=1$,  there is an essential surface $S^{*}\subset M_K$, with $\abs{\partial S^{*}}$  boundary components, 
and such that each component of $\partial S^{*}$ has slope $a^{*}/b^{*}$ and
$$\frac{-2\chi(S^{*})}{{\abs{\partial S^{*}} b}} \in jx^{*}_K.$$

\begin{define}  With the notation as above a  {\em Jones surface} of  $K$, is an essential surface $S\subset M_K$ such that, either
\begin{itemize} 
\item $\partial S$ represents a Jones slope  $a/b\in js_K$, with $b>0$ and $\gcd(a, b)=1$, and  we have
$$ \frac{2\chi(S)}{{\abs{\partial S} b}} \in jx_K; \ \ \  {\rm or}$$
\item$\partial S^{*}$ represents a Jones slope $a^{*}/b^{*}\in js^{*}_K$, with $b^{*}>0$ and $\gcd(a^{*}, b^{*})=1$, and we have  $$ \frac{-2\chi(S^{*})}{{\abs{\partial S^{*}} b^{*}}} \in jx^{*}_K.$$
\end{itemize}

\end{define}


\smallskip

\section{Jones surfaces of adequate knots} 
Let $D$ be a link diagram, and $x$ a crossing of $D$.  Associated to
$D$ and $x$ are two link diagrams,  called the \emph{$A$--resolution} and \emph{$B$--resolution} of
the crossing. See Figure \ref{resolutions}.
A Kauffman state $\sigma$ is a choice of $A$--resolution or
$B$--resolution at each crossing of $D$. The result of applying 
a state
$\sigma$ to $D$ is  a collection  $s_\sigma$
of disjointly embedded circles in the projection plane. 
We can encode the choices that lead
to the state $\sigma$ in a graph $G_\sigma$, as follows. The vertices
of $G_\sigma$ are in $1-1$ correspondence with the state circles of
$s_\sigma$. Every crossing $x$ of $D$ corresponds to a pair of arcs
that belong to circles of $s_\sigma$; this crossing gives rise to an
edge in $G_\sigma$ whose endpoints are the state circles containing
those arcs.

Given a Kauffman state $\sigma$  we construct a surface $S_\sigma$,
as follows. Each state circle of $\sigma$ bounds a disk in $S^3$. This
collection of disks can be disjointly embedded in the ball below the
projection plane. At each crossing of $D$, we connect the pair of
neighboring disks by a half-twisted band to construct a surface
$S_\sigma \subset S^3$ whose boundary is $K$. See Figure
\ref{resolve}.

\begin{figure}
  \includegraphics[scale=1.4]{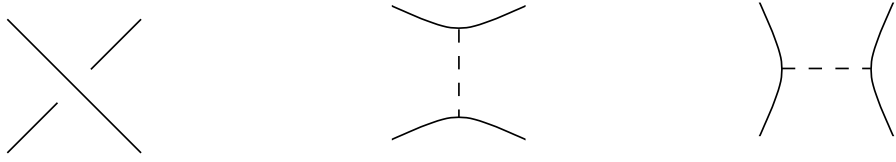}
  
  \caption{From left to right: A crossing, the $A$-resolution and  the
 the $B$-resolution.}
    
  \label{resolutions}
\end{figure}

\begin{figure}[ht]
\includegraphics[scale=.8]{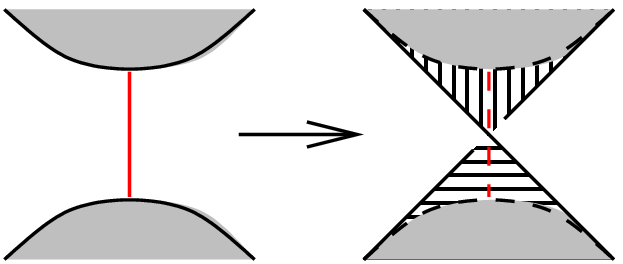}
\hspace{2cm}
\includegraphics[scale=.8]{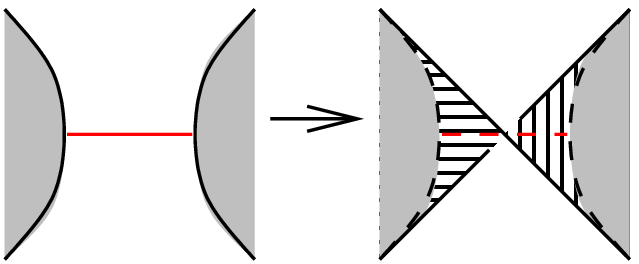}

\caption{The two resolutions of a crossing, the arcs recording them  and their contribution to state surfaces.} 
\label{resolve}
\end{figure}

\begin{define}\label{defi:adequate}
A link diagram $D$ is called \emph{$A$--adequate} if the state graph
$\GA$ corresponding to the all--$A$ state contains no 1--edge loops.
Similarly, $D$ is called \emph{$B$--adequate} if the all--$B$ graph $\GB$
contains no 1--edge loops.  A link diagram 
is \emph{adequate} if it is
both $A$-- and $B$--adequate.  A link that admits an adequate diagram
is also called \emph{adequate}.
\end{define}
It is known that  the number of negative crossings  $c_-(D)$  of an $A$--adequate knot diagram is a knot invariant.
Similarly, the number of positive crossings $c_+(D)$ of  a  $B$-adequate knot diagram is a knot invariant.
In fact, the crossing number of $K$ is realized by the adequate diagram; that is  we have $c(K)=c(D)=c_-(D)+c_+(D)$ \cite{Li}.
Let $v_A(D)$ and  $v_B(D)$ be the number of state circles in the all--$A$ (resp. all--$B$) state of the knot diagram $D$.
Also let $S_A=S_A(D)$ and $S_B=S_B(D)$ denote the surfaces corresponding to the all-$A$ and all-$B$ state of $D$.
The following  theorem summarizes known results about bounds on the degree of the colored Jones polynomials.
The first inequalities in both part (a) and (b) below are well known results that can be found, for example, in Lickorish's book
\cite[Lemma 5.4]{Li}. Inequalities \ref{lee1} and \ref{lee2},  that generalize and strengthen results of \cite{Kalee}, have been more recently established by Lee. See   \cite[Theorem 2.4]{lee}
or \cite{leethesis}.

\begin{theorem} 
\label{degLe}
Let $D$ be a diagram of a knot $K$.
\smallskip

\smallskip

(a) We have
$$4\, d_-[J_{K}(n)] \ge  -2c_- (D) n^2 + 2(c(D) -v_A(D)) n  +2 v_A(D) -2 c_+(D).$$
If $D$ is $A$--adequate, then  equality holds for all $n \ge 1$. Moreover, if $D$ is not $A$--adequate
then 
\begin{equation}
\label{lee1}
4\, d_-[J_{K}(n)] \ge  -2c_-(D) n^2 + 2(c(D) -v_A(D)+1) n  +e(n),
\end{equation}
\noindent where $e(n): \NN \to \QQ $ is a periodic function of $n$ with integral period.
\smallskip

\medskip

\medskip

\smallskip

(b) We have
$$
4 \, d_+[J_{K}(n)] \le 2c_+ (D)n^2 + 2(v_B(D) - c(D)) n +2 c_-(D)-  2v_B(D).
$$
If $D$ is $B$--adequate, then  equality holds for all $n \ge 1$. Moreover, if  $D$ is not  $B$--adequate
then 
\begin{equation}
\label{lee2}
4 \, d_+[J_{K}(n)] \le 2c_+(D) n^2 + 2(v_B(D) - c(D)-1) n +e^{*}(n),
 \end{equation}
\noindent where  $e^{*}(n): \NN \to \QQ $ is a periodic function of $n$ with integral period. \qed
\end{theorem}

The following theorem, which shows that the Strong Slope Conjecture is true for adequate knots,  was proven 
in \cite{Effie-Anh-slope} building on work   in  \cite{FKP, FKP-guts}.

\begin{theorem}\label{thm:essential} 
Let $D$ be an $A$--adequate diagram of a knot $K$. Then the 
surface $S_A$ is essential  in the knot
complement $M_K$, and it  has boundary slope $-2c_-$. Furthermore, we have
$$-2c_- = \lim_{n \to \infty} 4\, n^{-2}d_-[J_K(n)] \ \ \ {\rm  and } \ \ \  2\chi(S_A)= 2(v_A(D)-c(D)).$$

Similarly, if $D$ is a $B$--adequate diagram of a
knot $K$, then $S_B$ is essential in the knot 
complement $M_K$, and it  has boundary slope $-2c_+$.
 Furthermore, we have
$$2c_+ =  \lim_{n \to \infty} 4\, n^{-2}d_+[J_K(n)] \ \ \ {\rm and } \ \ \  2\chi(S_B)= 2(v_B(D)-c(D)). $$
In particular, if $K$is adequate, then it satisfies the Strong Slope Conjecture and $S_A$, $S_B$ are Jones surfaces.\qed
\end{theorem}

\section{ Colored Jones polynomials and adequate knots}

Let the notation be as in the last section. We recall that the {\emph Turaev genus} of a knot diagram $D=D(K)$ is  defined by
\begin{equation}
\label{Tur}
g_T(D)=(2-v_A(D)-v_B(D)+c(D))/2
\end{equation}
The Turaev genus of a knot $K$ is defined by
\begin{equation}
\label{TurK}
g_T(K)= {\rm min } \left\{ g_T(D)\ | \   D=D(K) \right \}
\end{equation}
The genus $g_T(D)$ is the genus of the {\emph Turaev surface} $F(D)$,  corresponding to $D$. This surface is constructed as follows:
Let $\Gamma \subset S^2$ be the planar, 4--valent graph of the 
diagram $D$.  Thicken the (compactified) projection plane to $S^2 \times
[- 1, 1]$, so that $\Gamma$ lies in $S^2 \times \{0\}$. Outside a
neighborhood of the vertices (crossings),  $\Gamma \times [- 1, 1]$ will be part of $F(D)$. 

In the neighborhood of
each vertex, we insert a saddle, positioned so that the boundary
circles on $S^2 \times \{1\}$ are the
components
of the $A$--resolution and the boundary circles on $S^2
\times \{- 1\}$ are the components of  the  $A$--resolution.
See Figure
\ref{saddle}. Then, we cap off each circle with a disk, obtaining
a  closed surface $F(D)$.

\begin{figure}[ht]
\begin{center}
\includegraphics[scale=1.2]{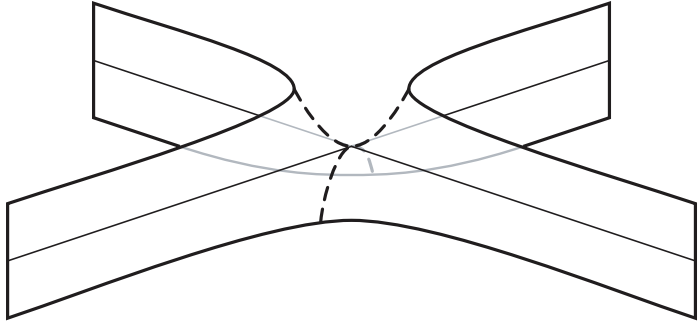}
\end{center}
\caption{A saddle between components of the $A$, $B$-resolutions near a vertex of the 4-valent graph $\Gamma$ corresponding to crossing of $D$. The portion of $\Gamma$ on the saddle,  is indicated in solid line. 
The dashed lines indicate the edges of the state graphs $\GA$, $\GB$ corresponding to the crossing. } 
\label{saddle}
\end{figure}

The surface $F(D)$ has the following properties: 
\smallskip

(i) It  is a  Heegaard surface of ${S}^{3}$.
\smallskip

 (ii) $D$ is alternating  
on $F(D)$; in particular $D$ is an alternating diagram if and only if $g_T(F(D))=0$.
\smallskip

(iii) the 4-valent graph underlying $D$ defines a cellulation of $F(D)$ for which the 2-cells can be colored in a 
checkerboard fashion.
\medskip

We warn the reader that these properties are not enough to characterize the Turaev surface $F(D)$.
There reader is referred to \cite{DFK} or to a survey article by Champanerkar and Kofman  \cite{CK} for more details.
\vskip 0.06in

We will need the following result of Abe \cite[Theorem 3.2]{Abe}.
\begin{theorem} \label{abe} Suppose that $D$ is an adequate diagram of a knot $K$. Then we have

$$g_T(K)=g_T(D)=(2-v_A(D)-v_B(D)+c(D))/2.$$ \qed
\end{theorem}
We are now ready to state and prove the main result of this paper, which implies Theorem \ref{mainintro} stated in the introduction.
\begin{theorem} \label{main}  For a knot  $K$  let 
$js_K$, $ js^*_K $ and $jx_K, jx^{*}_K$ be the sets associated to $J_K(n)$ as in Definitions  \ref{jslopes} and \ref{jchar}. Also let $c(K)$  and  $g_T(K)$ denote the crossing number and the Turaev genus of $K$, respectively.
Then, $K$ is adequate if and only if the following are true:

\begin{enumerate}
\smallskip

\item  There are Jones slopes $s\in js_K$ and $s^{*} \in  js^{*}_K$, with $s-s^{*} =2c(K);$  and
\smallskip

\item there are  $x\in jx_K$ and $x^*\in  jx^{*}_K$ with  $x-x^* =2 (2-2g_T(K)-c(K)).$
\smallskip

\end{enumerate}

\noindent Furthermore, any diagram of $K$ that realizes  $c(K)$ is adequate and it also realizes $g_T(K)$.
\end{theorem}

\begin{proof} Suppose that $K$ is a knot with an adequate diagram $D=D(K)$.  We know that $c(K)=c(D)=c_+(D)+c_-(D)$.
 By Theorem \ref{degLe}, equation \ref{Tur}, and Theorem \ref{abe} we have 
 
\begin{eqnarray*}
 4 \, d_+[J_{K}(n)] -4 \, d_-[J_{K}(n)]  =& \\
 2c(K)n^2+2(v_B(D)+v_A(D)-2c(D)) n+ 2(c(D)-v_B(D)-v_A(D)) = &  \\
  2c(K)n^2+2(2-2g_T(D)-c(D)) n+ 2(c(D)-v_B(D)-v_A(D))=&\\
  2c(K)n^2+2(2-2g_T(K)-c(K)) n+ 2(2g_T(K)-2).
\end{eqnarray*}
\noindent where the last equations follows from the fact that, since $D$ is adequate,  by Theorem
 \ref{abe} we have $g_T(D)=g_T(K)$. Thus the quantities $s=2c_+(D)$, $s^{*}=-2c_-(D)$,  $ x=2(v_A(D)-c(D))$ and $x^{*}=2(c(D)-v_B(D))$
 satisfy the desired equations.

Conversely, suppose that  we have $s, s^{*}, x, x^{*}$ as in the statement above and let
$p=p(K)$ denote the common period
of $4 \, d_-[J_{K}(n)] ,  4 \, d_+[J_{K}(n)]$.

There is $0\leq i \leq p$, such that
for infinitely many $n>>0$ we have
$$a(n)=s \ \ \ {\rm and} \ \ \   a^{*}(n+i)=s^{*}. $$
Let $D$ be a diagram of $K$ that realizes the crossing number $c(K)$. Let $c_+(D)$ and $c_-(D)$ denote the number of positive and negative crossings 
in $D$, respectively.  By applying Theorem \ref{degLe} to $D$ we must have
$$s n^2 +o(n)\leq  2c_+(D) n^2+o(n),$$
for infinitely many $n>>0$; hence we obtain
$s\leq 2c_+(D)$.
Similarly, we have
$$s^{*} (n+i)^2 +o(n)\leq  2c_-(D) (n+i)^2+o(n),$$
and we 
conclude that and $- s^{*}\leq 2c_-(D)$. Since by assumption
$s-s^{*} =2c(K)$, and $c_-(D)+c_+(D)=c(K)$ we conclude that 
\begin{equation}
\label{s=c}
s=2c_+(D)\ \ \ {\rm and} \ \ \ -s^{*}= 2c_-(D).
\end{equation}

To continue recall that by assumption, there is  $0\leq j\leq p$, such that
for infinitely many $n>>0$ we have
\begin{equation}
\label{b=x}
b(n)=x \ \ \ {\rm and} \ \ \   b^{*}(n+j)=x^{*}.
\end{equation}
Now, by equation \ref{b=x}, and using  Theorem  \ref{degLe} as above,  we obtain
$$x \leq  2(v_B(D) - c(D)).$$ Similarly, using that $-s^{*}= 2c_-(D)$, we get that
for infinitely many $n>>0$ we have
\begin{equation}
-x^{*}(n+j)+ 4n c_- (D)\leq - 2(c(D) -v_A(D))(n+j)+ 4n c_-(D).
\end{equation}
Hence we obtain
$-x^{*}\leq  2(c(D) -v_A(D))$. This in turn, combined with equation \ref{Tur},  gives
\begin{equation}
\label{gTD}
x-x^*\leq 2(v_B(D)+v_A(D)-2c(D))=
2(2-2g_T(D)-c(D)). 
\end{equation}
On the other hand, by assumption, 
\begin{equation}
\label{gTK}
x -x^{* }=2(2-2g_T(K)-c(K)).
 \end{equation}
 Since $g_T(K)\leq g_T(D)$ and $c(D)=c(K)$, by equations \ref{gTD} and \ref{gTK},  we conclude that 
 $$g_T(K)= g_T(D), \ \ 
 x=2c(D)-2v_B \ \ {\rm and} \ \  x^*=2v_A-2c(D).$$
 This in turn implies that,  for infinitely many $n>n_K$,
  we have

 \begin{equation}
 \label{last1}
 4\, d_-[J_{K}(n)] =  -2c_-(D) n^2 + 2(c (D)-v_A(D))n  +f(D),
 \end{equation}
 
  and
  
   \begin{equation}
 \label{last2}
4\, d_+J_{K}(n)] =2c_+(D) n^2 + 2(c(D) -v_B(D)) n  +f^{*}(D),
\end{equation}
 
\noindent  where $f(D), f^{*}(D)$ are periodic functions of $n$. It follows that 
 $f(D), f^{*}(D)$ can take at most finitely many distinct values and that
 they are bounded by a universal constant.
   Now  Theorem
 \ref{degLe}  implies 
  that $D$ has to be both $A$ and $B$ adequate; hence adequate. For, otherwise one of inequalities \ref{lee1}, \ref{lee2} would have to hold, which would contradict
  equations \ref{last1}, \ref{last2}.
  
  To finish the proof of the theorem notice that the arguments above imply that if $K$ is a knot for which (1), (2) are satisfied and $D$
  is diagram of $K$ that realizes $c(K)$ then $D$ is adequate. Thus, by Theorem \ref{abe}, we also have $g_T(D)=g_T(K)$.

\end{proof}

 Now we explain how Theorem \ref{mainintro} follows: 
 \medskip

 {\em Proof of Theorem \ref{mainintro}.} First suppose that $K$ is a knot with an adequate diagram $D$. Then $c(K)=c(D)$.
 The first part of the proof of Theorem \ref{main} implies that equations \ref{first} are satisfied for all $n>0$.
 Suppose conversely that for some $n>n_K$, equations \ref{first} are satisfied. Since $s_1(n), s_2(n)$ are periodic with integral period we conclude that there must be
 infinitely many $n>>0$  for which equations \ref{first} are true. Taking $D$ a knot diagram of $K$ that realizes $c(K)$, the argument in the second proof of Theorem \ref{main} implies that $g_T(D)=g_T(K)$
 and that equations \ref{last1} and \ref{last2} hold for $D$. Hence as before $D$ is adequate.\qed
\medskip

Theorem \ref{thm:essential} implies that the Strong Slope Conjecture is true for adequate knots. The next result implies that for knots that satisfy
the conjecture the characterization provided by 
 Theorem \ref{main} can be expressed in terms of properties of their spanning surfaces.

\begin{theorem} \label{mainjones } Given a knot $K$ with  crossing number $c(K)$ and  Turaev genus    $g_T(K)$ the following are equivalent:
\smallskip
\begin{enumerate}
\item  $K$ is adequate.
\smallskip

\smallskip

\item There are Jones surfaces $S$ and $S^*$
with boundary slopes $s,s^*$ such that:
\begin{equation}
\label{surfaces}
s-s^*= 2c(K) \ \ {\rm and} \ \      \frac{\chi(S)}{{\abs{\partial S}}} + \frac{\chi(S^{*})}{{\abs{\partial S^{*}} }}+c(K) =2-2g_T(K).
\end{equation}
\smallskip
\item There are Jones surfaces $S$ and $S^*$, that
 are in addition spanning surfaces of $K$ (i. e. that is $\partial S=\partial S^{*}=K$)
 such that
\begin{equation}
\label{surfacesspa}
s-s^*= 2c(K) \ \  {\rm and} \ \      \chi(S)+ \chi(S^{*})+c(K) =2-2g_T(K).
\end{equation}
 \end{enumerate}
\end{theorem}

\begin{proof} Suppose that $K$ is adequate. Then by
Theorem \ref{thm:essential},  and the calculation in the beginning of the proof of Theorem \ref{main},  the state surfaces $S_A$ and $S_B$ obtained from any adequate diagram of $K$ satisfy 
equations \ref{surfaces}. In fact, in this case, we have 
we have  $\abs{\partial S}=\abs{\partial S^{*}}=1$.

Conversely, suppose that there are Jones surfaces $S$ and $S^*$  with boundary slopes $s,s^*$ such that
$s-s^*= 2c(K)$.  By the proof of Theorem \ref{main}, if $D$ is a diagram realizing $c(K)$, we have
$s=2c_+(D)$ and $s^{*}=-2c_-(D)$.
Since  $S$ and $S^*$  have integral slopes,  the number of sheets in each of them is one; thus $b=b^{*}=1$.
Since  $S$ and $S^*$  are Jones surfaces, we have

$$x= \frac{2\chi(S)}{{\abs{\partial S}}} \in  jx_K \ \ \ {\rm and} \ \ \ x^{*}=\frac{- 2\chi(S^{*})}{{\abs{\partial S^{*}} }}\in  jx^{*}_K.$$
Thus we get 
$$x-x^{*}=2(2-2g_T(K)-c(K)).$$
Thus by Theorem \ref{main}, $D$ must be an adequate diagram of $K$.
This shows that (1) and (2) are equivalent.

Now (3), clearly implies (2). Finally, since,  by Theorem \ref{main}, (2)  implies that $K$ is adequate, and (3) is true for adequate knots, we get that (2) implies (3).
\end{proof}

\section{Alternating knots}
Recall that
a knot  $K$ is alternating if and only if $g_T(K)=0$ \cite{DFK}. Combining this with Theorem \ref{main} we will show the following.

\begin{corollary} \label{mainalter}  For a knot  $K$  let 
$js_K$, $ js^*_K $ and $jx_K, jx^{*}_K$ be the sets associated to $J_K(n)$ as in Definitions \ref{jslopes} and \ref{jchar}. Then, $K$ is alternating if and only if the following are true:
\begin{enumerate}
\item  There are Jones slopes $s\in js_K$ and $s^{*} \in  js^{*}_K$, with $s-s^{*} =2c(K);$ and
\medskip

\item there are  $x\in jx_K$ and $x^*\in  jx^{*}_K$ with $x-x^* =4-2c(K).$

\end{enumerate}

\end{corollary}
\begin{proof}
If $K$ is alternating then Theorem \ref{main} and the fact that $g_T(K)=0$ imply that (1) and (2) hold. Conversely suppose that we have $s, s^{*}, x, x^{*}$ as in the statement above and let $D$ be a diagram of $K$ such that $c(D)=c(K)$.
Let $g_T(D)$ denote the Turaev genus of $D$. The argument in the proof of Theorem \ref{main} implies that

$$x-x^{* }=4-2c(K)\leq 4-2g_T(D)-2c(K),$$
which can only hold if $g_T(D)=0$ and hence $D$ is alternating.
\end{proof}

Recently, Howie  \cite{howie}  and independently Greene \cite{greene} gave characterizations of alternating knots in terms of properties of spanning surfaces.
In particular \cite[Theorem 2]{howie} states that a non-trivial knot $K$ is alternating if and only if it admits 
spanning surfaces  $S$ and $S^*$, such that the following holds.
$$
\chi(S)+\chi(S^{*})+  i(\partial S, \partial S^{*})/2=2  \ \ \ {\rm and} \ \ \  i(\partial S, \partial S^{*})/2=c(K),
$$

\noindent where $ i(\partial S, \partial S^{*})$ denotes the geometric intersection number
of $\partial S, \partial S^{*}$ on the $\partial M_K$. Note, we assume that this intersection number is minimal in the isotopy classes of  $\partial S, \partial S^{*}$.

Combining  this with Theorem \ref{mainjones } 
we obtain the following.
\begin{corollary} Given a non-trivial  knot $K$ with crossing number $c(K)$ the following are equivalent.
\begin{enumerate}
\item $K$ is alternating.
\item There are Jones surfaces $S$ and $S^*,$ that are spanning surfaces of $K$ 
with boundary slopes $s,s^*$ such that
 $$ \chi(S)+\chi(S^{*})+ (s-s^{*})/2=2  \ \  {\rm and} \ \  s-s^{*}= 2c(K) $$
\item  There are spanning surfaces  $S$ and $S^*$ of $K$
such that
$$ \chi(S)+\chi(S^{*})+i(\partial S,  \partial S^{*})/2=2  \ \ \ {\rm and} \ \ \  i(\partial S,  \partial S^{*})=2c(K).$$

\end{enumerate}
\end{corollary}
\begin{proof}
Suppose that  $D$  is a reduced alternating diagram for  a knot $K$. Then,
$D$ is both $A$ and $B$-adequate and $g_T(D)=g_T(K)=0$.
Now the checkerboard surfaces, $S, S^{*}$ of $D$ are the all-$A$ and all-$B$ state surface  which are Jones surfaces.
Thus they satisfy the desired properties.

Conversely suppose that we have Jones surfaces $S$ and $S^{*}$ as above and let $D$ be a diagram of $K$ that realizes
$c(K)$. Set
$$ x= 2\chi(S) \ \ \ {\rm and} \ \ \ x^{*}= -2\chi(S^{*})
$$

By hypothesis,
\begin{equation}
\label{last11}
x-x^{*}=4-2c(K)
\end{equation}
By Theorem \ref{degLe}, as in the proof of Theorem \ref{main},
we get 
\begin{equation}
\label{last12}
x-x^*\leq 2(v_B(D)+v_A(D)-2c(D))=
2(2-2g_T(D)-c(D)).
\end{equation}
Since $c(D)=c(K)$, combining  equations \ref{last11} and \ref{last12}  we have
$$4-2c(K)\leq 2(2-2g_T(D)-c(K)), $$
which gives $g_T(D)=0$. This in turn implies that $D$ is alternating \cite{DFK}.
Thus we showed that (1) and (2) are equivalent. 

Now we explain how (2) implies (3): Let $D$ be a diagram of $K$ that realizes $c(K)$; as above $D$ is alternating.
The argument in the proof  of Theorem \ref{main} shows that 
if we have Jones slopes $s, s^{*}$  as above such  that $2c(K)=s-s^{*}$ then 
$s=2c_+(D)$ and $s^{*}=-2c_-(D)$. Suppose that the simple closed curves $\partial S, \partial S^{*}$ have been isotoped on the torus $\partial n(K)$ to minimize their intersection number.
 Then we have
$i(\partial S, \partial S^{*})=2c_+(D)-(-2c_-(D)) =2c(K)=s-s^{*}$; thus (3) follows.
Hence (1) above implies (3).  

Finally, by  \cite[Theorem 2]{howie}, 
(3) implies (1). 
\end{proof}

\bibliographystyle{plain} \bibliography{biblio}

\begin{thebibliography}{10}

\bibitem{Abe}
Tetsuya Abe.
\newblock The {T}uraev genus of an adequate knot.
\newblock {\em Topology Appl.}, 156(17):2704--2712, 2009.

\bibitem{armond}
Cody Armond.
\newblock The head and tail conjecture for alternating knots.
\newblock {\em Algebr. Geom. Topol.}, 13(5):2809--2826, 2013.

\bibitem{ADK}
Cody Armond, Nathan Druivenga, and Thomas Kindred.
\newblock Heegaard diagrams corresponding to {T}uraev surfaces.
\newblock {\em J. Knot Theory Ramifications}, 24(4):0218--2165, 2015.

\bibitem{CK}
Abhijit Champanerkar and Ilya Kofman.
\newblock A survey on the {T}uraev genus of knots.
\newblock {\em Acta Math. Vietnam.}, 39(4):497--514, 2014.

\bibitem{DFK}
Oliver~T. Dasbach, David Futer, Efstratia Kalfagianni, Xiao-Song Lin, and
  Neal~W. Stoltzfus.
\newblock The {J}ones polynomial and graphs on surfaces.
\newblock {\em Journal of Combinatorial Theory Ser. B}, 98(2):384--399, 2008.

\bibitem{dasbach-lin:head-tail}
Oliver~T. Dasbach and Xiao-Song Lin.
\newblock On the head and the tail of the colored {J}ones polynomial.
\newblock {\em Compositio Math.}, 142(5):1332--1342, 2006.

\bibitem{fkp:filling}
David Futer, Efstratia Kalfagianni, and Jessica~S. Purcell.
\newblock {Dehn filling, volume, and the Jones polynomial}.
\newblock {\em J. Differential Geom.}, 78(3):429--464, 2008.

\bibitem{FKP}
David Futer, Efstratia Kalfagianni, and Jessica~S. Purcell.
\newblock Slopes and colored {J}ones polynomials of adequate knots.
\newblock {\em Proc. Amer. Math. Soc.}, 139:1889--1896, 2011.

\bibitem{FKP-guts}
David Futer, Efstratia Kalfagianni, and Jessica~S. Purcell.
\newblock {\em Guts of surfaces and the colored {J}ones polynomial}, volume
  2069 of {\em Lecture Notes in Mathematics}.
\newblock Springer, Heidelberg, 2013.

\bibitem{fkp:survey}
David Futer, Efstratia Kalfagianni, and Jessica~S. Purcell.
\newblock {J}ones polynomials, volume, and essential knot surfaces: a survey.
\newblock {\em Proceedings of Knots in Poland III, Banach Center Publications},
  100:51--77, 2014.

\bibitem{fkp:qsf}
David Futer, Efstratia Kalfagianni, and Jessica~S. Purcell.
\newblock Quasifuchsian state surfaces.
\newblock {\em Trans. Amer. Math. Soc.}, 366(8):4323--4343, 2014.

\bibitem{FKP-semi}
David Futer, Efstratia Kalfagianni, and Jessica~S. Purcell.
\newblock Hyperbolic semi-adequate links.
\newblock {\em Comm. Anal. Geom.}, 23(5):905--941, 2015.

\bibitem{ga-quasi}
Stavros Garoufalidis.
\newblock The degree of a {$q$}-holonomic sequence is a quadratic
  quasi-polynomial.
\newblock {\em Electron. J. Combin.}, 18(2):Paper 4, 23, 2011.

\bibitem{ga-slope}
Stavros Garoufalidis.
\newblock The {J}ones slopes of a knot.
\newblock {\em Quantum Topol.}, 2(1):43--69, 2011.

\bibitem{greene}
Joshua~Evan Greene.
\newblock Alternating links and definite surfaces.
\newblock arXiv:1511.06329.

\bibitem{hatcher}
Allen~E. Hatcher.
\newblock On the boundary curves of incompressible surfaces.
\newblock {\em Pacific J. Math.}, 99(2):373--377, 1982.

\bibitem{howie}
Joshua Howie.
\newblock A characterisation of alternating knot exteriors.
\newblock arXiv:1511.04945.

\bibitem{Kalee}
Efstratia Kalfagianni and Christine Ruey~Shan Lee.
\newblock On the degree of the colored {J}ones polynomial.
\newblock {\em Acta Math. Vietnam.}, 39(4):549--560, 2014.

\bibitem{Effie-Anh-slope}
Efstratia Kalfagianni and Anh~T. Tran.
\newblock Knot cabling and degrees of colored jones polynomials.
\newblock {\em New York Journal of Mathematics}, Volume 21:905--941, 2015.

\bibitem{Ka}
Louis~H. Kauffman.
\newblock State models and the {J}ones polynomial.
\newblock {\em Topology}, 26(3):395--407, 1987.

\bibitem{leethesis}
Christine Ruey~Shan Lee.
\newblock {J}ones-type link invariants and applications to 3-manifold topology.
\newblock MSU, PhD Thesis, 2015.

\bibitem{lee}
Christine Ruey~Shan Lee.
\newblock Stability properties of the colored {J}ones polynomial.
\newblock arXiv:1409.4457.

\bibitem{Li}
W.~B.~Raymond Lickorish.
\newblock {\em An introduction to knot theory}, volume 175 of {\em Graduate
  Texts in Mathematics}.
\newblock Springer-Verlag, New York, 1997.

\bibitem{lick-thistle}
W.~B.~Raymond Lickorish and Morwen~B. Thistlethwaite.
\newblock Some links with nontrivial polynomials and their crossing-numbers.
\newblock {\em Comment. Math. Helv.}, 63(4):527--539, 1988.

\bibitem{murasugitait}
Kunio Murasugi.
\newblock Jones polynomials and classical conjectures in knot theory.
\newblock {\em Topology}, 26(2):187--194, 1987.

\bibitem{Thistle}
Morwen~B. Thistlethwaite.
\newblock A spanning tree expansion of the jones polynomial.
\newblock {\em Topology}, 26(3):297--309, 1987.

\bibitem{thi:adequate}
Morwen~B. Thistlethwaite.
\newblock On the {K}auffman polynomial of an adequate link.
\newblock {\em Invent. Math.}, 93(2):285--296, 1988.

\bibitem{turaevs}
Vladimir~G. Turaev.
\newblock A simple proof of the {M}urasugi and {K}auffman theorems on
  alternating links.
\newblock {\em Enseign. Math. (2)}, 33(3-4):203--225, 1987.

\end{thebibliography}
\end{document}